\documentclass[11pt,a4paper]{article}
\usepackage{a4wide,amsfonts,amsmath,latexsym,amsthm,amssymb,euscript,eufrak,graphicx,units,mathrsfs,setspace,stmaryrd,xcolor}

\usepackage{float}
\newfloat{figure}{H}{lof}
\floatname{figure}{\figurename}

\usepackage[colorlinks=true,
            linkcolor=blue,
            citecolor=blue,
            urlcolor=blue]{hyperref}
\usepackage[french,english]{babel}
\usepackage[T1]{fontenc}
\usepackage[utf8]{inputenc}

\DeclareMathAlphabet{\eufrak}{U}{}{}{}  % Euler fraktur math
\SetMathAlphabet\eufrak{normal}{U}{euf}{m}{n}
\SetMathAlphabet\eufrak{bold}{U}{euf}{b}{n}

\newtheorem{prop}{Proposition}[section]

\newtheorem{theorem}[prop]{Theorem}
\newtheorem{lemma}[prop]{Lemma}
\newtheorem{corollary}[prop]{Corollary}

\newtheorem{assumption}[prop]{Assumption}

\theoremstyle{definition}
\newtheorem{example}[prop]{Example}
\newtheorem{remark}[prop]{Remark}

\newtheorem{definition}[prop]{Definition}

\numberwithin{equation}{section}

\def\E{\mathbb{E}}
\def\P{\mathbb{P}}
\def\real{\mathbb{R}}

\def\1{\textbf{1}}

\newcommand{\eps}{\varepsilon}

\newcommand{\be}{\begin{equation}}
\newcommand{\ee}{\end{equation}}
\newcommand{\bde}{\begin{displaymath}}
\newcommand{\ede}{\end{displaymath}}
\newcommand{\beq}{\begin{eqnarray*}}
\newcommand{\eeq}{\end{eqnarray*}}
\newcommand{\beqa}{\begin{eqnarray}}
\newcommand{\eeqa}{\end{eqnarray}}
\newcommand{\bel }{\left\{\begin{array}{ll}}
\newcommand{\eel}{\cr \end{array} \right.}

\newcommand{\bex}{\begin{ex} \rm }
\newcommand{\eex}{\end{ex}}

\def\R{\mathbb R}

\def\E{\mathbb E}

\def\P{\mathbb P}

\definecolor{ying}{rgb}{0.8, 0.0, 0.04}

\DeclareSymbolFontAlphabet{\mathrsfs}{rsfs}

\author{Caroline Hillairet\footnote{ENSAE  IP Paris, CREST UMR 9194, 5 avenue Henry Le Chatelier, 91120 Palaiseau, France.  Email: \texttt{caroline.hillairet@ensae.fr}} \and Thibaut Mastrolia\footnote{Department of Industrial Engineering and Operations Research UC Berkeley, 4141 Etcheverry Hall, Berkeley, CA 94720, United States.  Email: \texttt{mastrolia@berkeley.edu}} \and Anthony R\'eveillac\footnote{INSA de Toulouse, IMT UMR CNRS 5219, Universit\'e de Toulouse, 135 avenue de Rangueil 31077 Toulouse Cedex 4, France. \; Email: \texttt{anthony.reveillac@insa-toulouse.fr}} }

\title{Martingale Representation in Counting-Process Filtrations via Poisson Embedding  \footnote{The authors acknowledge the Banff International Research Station for its support during the Focused Research Group 26frg015 ``Stochastic Control of Hawkes-Driven Dynamical Systems with Applications to Cyber Risk Management". Caroline Hillairet's research benefited from the support
of the Chair Stress Test, RISK Management and Financial Steering, led by the French Ecole
Polytechnique and its Foundation and sponsored by BNP Paribas.}}

\begin{document}
\maketitle

\begin{abstract}
We revisit martingale representation for counting processes using the Poisson embedding, deriving the representation in their natural filtration from that of the underlying Poisson random measure. Our approach combines Poisson Malliavin calculus with the Clark-Ocone and Mecke formulas. We first characterize measurability with respect to the counting-process filtration through the structure of the Malliavin derivative. We show that any random variable measurable in the counting-process filtration is invariant with respect to the auxiliary mark on the acceptance region obtained from the Poisson embedding method and vanishes on the rejection region. We then prove that every square-integrable martingale in this smaller filtration consequently admits a unique stochastic integral representation with respect to the compensated counting process: the integrand is explicitly characterized  as the  conditional expectation of the  Malliavin derivative, which identifies as the conditional effect of an additional accepted event on the terminal functional. Additionally, we define the  mirroring  counting process which counts the rejected events in the Poisson embedding procedure and we then provide an orthogonal decomposition of any martingale in the joint filtration of the accepted and rejected counting processes as a sum of a martingale in the counting filtration obtained by Poisson embedding and its mirroring martingale part on the rejected domain.   Finally, we apply these results to general Hawkes processes constructed by Poisson embedding, obtaining a martingale representation and an explicit integrand formula in their natural filtration.
\end{abstract}

\section{Introduction}

 A martingale is a process whose current value is the best predictor, in the mean-square sense, of its future value given the information currently available. Martingale representation theorems express its evolution as a stochastic integral with a predictable integrand with respect to the driving noises. Beyond their structural role in stochastic
analysis, these representations are central to stochastic control, representation of Backward Stochastic Differential Equation or hedging strategy in finance. 
In particular, they allow backward equations to be formulated using
the observable sources of randomness with some closed-form for the integrand by using Malliavin calculus, known as Clark-Ocone formula. For systems driven by point
processes, this leads naturally to integration with respect to compensated
counting processes. Such formulations are particularly useful in
non-Markovian control problems, where the intensity of a point process (if it exists)   and the objective criterion related to an optimization problem
may depend on the entire observed history; see, for instance,
\cite{confortola2014backward}.

The representation theory for point-process filtrations is classical and
Jacod~\cite{jacod1975multivariate} established general results for multivariate
point processes, subsequently developed in the framework of
Jacod and Shiryaev~\cite{jacod2013limit}. 
More precisely, Theorem~III.4.37 of the latter reference concerns
a marked point process with random measure $\mu^X$ on $ \mathbb R_+ \times E$ (with $(E, \mathcal E)$ an auxiliary Blackwell space) and predictable
compensator $\nu^X$, and  considers the smallest right-continuous filtration
making $\mu^X$ optional.
 In this filtration, every local martingale admits
a representation 
\[
    P_t=P_0+
    \int_{(0,t]\times E}W_s(x)\,
    (\mu^X-\nu^X)(\mathrm{d}s,\mathrm{d}x),
\]
for a predictable integrand $W$ such that for $x\in E$, $W_\cdot(x)$ is locally integrable. 
In particular, if $C$ is a non-explosive simple counting process with
predictable intensity $\lambda$ in its natural filtration $\mathbb F^C$,
every square-integrable $\mathbb F^C$-martingale $P$ admits the representation
\[
 P_t=P_0+\int_0^t U_s\bigl(dC_s-\lambda_s\,ds\bigr),
 \text{ with } 
 \mathbb E\!\left[\int_0^T |U_s|^2\lambda_s\,ds\right]<\infty,
\]
for any finite time horizon $T>0$,
where $U$ is $\mathbb F^H$-predictable and unique up to
$\lambda_s(\omega)\,ds\,\mathbb P(d\omega)$-almost everywhere equality.
The development of this theory in the 1970s includes the work \cite{ching2006representation, boel1975martingales}  alongside Jacod's general treatment of
multivariate point processes~\cite{jacod1975MRP}; see also
Jacod and Shiryaev~\cite[Chapter~III, Section~4]{jacod2013limit}.
These results formulate representation intrinsically in terms of a point
process, its filtration, and its compensator. This abstract framework however
does not itself specify an underlying Poisson noise nor identify the
integrand through the effect of perturbing that noise. Although it already
supports changes of measure, additional structure
is needed to connect the representation to a fixed Poisson reference
construction in order to build a tractable stochastic control theory, especially in weak formulation with change of measures indexed on a reference benchmark, see for example \cite{confortola2014backward}.

We provide this connection for counting processes constructed by Poisson
embedding. On a canonical space equipped with the Poisson law $\mathbb P$,
let $\mathbf N$ be a Poisson random measure with intensity $dt\,d\theta$.
We consider the thinning equation
\[
C_t=\int_{(0,t]\times\mathbb R_+}
       \mathbf 1_{\{\theta\leq\lambda_s\}}\,\mathbf N(ds,d\theta),
 \qquad
 \lambda_t=\Lambda\bigl(t,C|_{[0,t)}\bigr),
\] where $\Lambda$ is a positive function of the time $t$ and the path on  $[0,t)$ of the point process $C$, with integrability and pathwise well-posedness
assumptions specified in Section~\ref{sec:problem}; see Daley and Vere-Jones \cite[Chapter 14.7]{daley2008introduction} for this general construction or 
Br\'emaud and Massouli\'e~\cite{bremaud1996stability} for this construction
in the Hawkes setting. The observed filtration $\mathbb F^C$ is smaller
than the Poisson filtration $\mathbb F^{\mathbf N}$: it records accepted
event times, but neither rejected points nor auxiliary marks.
Consequently, the Poisson representation theorem alone does not identify
a representation integrand predictable with respect to the observed
history. We fix this issue by using Malliavin calculus for point processes, see \cite{Last2016, Privault_2009}. The advantage of our approach is twofold. First, we deduce the martingale
representation property in $\mathbb F^C$ directly from the predictable
representation property of $\mathbf N$ in $\mathbb F^{\mathbf N}$,
without invoking a general point-process representation theorem as in \cite{jacod2013limit}.
We show that the Poisson integrand associated with a terminal functional
of $C$ vanishes on rejected points and reduces on accepted points to an
$\mathbb F^C$-predictable process. Second, Malliavin calculus gives access to the representation integrand.
For a terminal functional $F$ of $C$, the Poisson difference operator
$D_{t,\theta}F:=F\circ\varepsilon^+_{t,\theta}-F$ measures the effect
of inserting an atom at $(t,\theta)$. Rejected atoms have no effect,
whereas the effect of an accepted atom does not depend on its auxiliary
mark. Combining this structure with the Poisson Clark-Ocone
formula~\cite{last2017penrose} yields, for
square integrable and $\mathcal F_T^C$-measurable random variable $F$ in the Poisson
Malliavin-Sobolev domain,
\[
 Z_{t,\theta}
   =\mathbb E[D_{t,\theta}F\mid\mathcal F^{\mathbf N}_{t-}]
     =\mathbf 1_{\{\theta\leq\lambda_t\}}U_t,\text{ with }
 U_t=\mathbb E[D_{t,0}F\mid\mathcal F^C_{t-}],
\]
where $Z$ is the integrand with respect to the compensated Poisson process
$\widetilde{\mathbf N}=\mathbf N-dt\,d\theta$, and $U$ is the integrand
with respect to  the compensated point process $dC_t-\lambda_t\,dt$. Thus $U$ is identified as the conditional
effect of an additional accepted event on the terminal functional.
The representation itself holds for every square-integrable
$\mathbb F^C$-martingale, without this additional Malliavin regularity.
This complements the Poisson-based Malliavin analysis
in~\cite{HHKR}, the pseudo-chaotic expansions
of~\cite{hillairet2024chaotic}, and the non-compensated Clark-Ocone
formula of~\cite{hillairet2024poisson}.

Hawkes processes~\cite{Hawkes,hawkes1974cluster} provide a natural example.
For a general Hawkes process $H$ with empty initial history,
\[
 \lambda_t=\mu+\Psi\Big(\int_{(0,t)}\Phi(t-s)\,dH_s\Big),
 \qquad \mu>0,
\]
where $\Psi:\mathbb R\longrightarrow \mathbb R_+$ and a kernel function $\Phi$ satisfy the assumptions of
Section~\ref{sec:application}, inserting an accepted event changes both the observed path
and the intensity of subsequent events. The Malliavin derivative
therefore accounts for this propagation through the event history.
Our general theorem applies with $C=H$, yielding the representation
in $\mathbb F^H$ and the explicit identification of its integrand.
This supplies a basis for studying backward equations and control
problems with history-dependent intensities. The structure of the paper is the following. Section~\ref{sec:problem} presents the framework, Section~\ref{sec:FN} proves the representation
theorem and Section~\ref{sec:application} treats the particular example of Hawkes processes. In addition, Section \ref{sec:ortho} provides 
a decomposition result of martingales in the joint filtration of the accepted and rejected counting processes into two orthogonal martingales measurable with respect to C or its counting complement  denoted by $\bar C,$ which counts the rejected events in the thinning procedure
\section{Problem setting and Malliavin calculus for point processes}\label{sec:problem}

\subsection{Poisson embedding and thinning construction of a point process}
\label{subsec:hawkes-thinning}

Let $\mathbb R_+:=[0,\infty)$ and
$\mathbb X:=\mathbb R_+\times\mathbb R_+$.
We fix a finite horizon $T>0$. The coordinates $x=(t,\theta) \in \mathbb X$ are interpreted as
time and an auxiliary mark. We consider the configuration space
\[
\Omega
:=
\left\{
\omega=\sum_{i\geq 1}\delta_{(t_i,\theta_i)}:
\omega(K)<+\infty
\text{ for every compact }K\subset\mathbb X
\right\},
\]
endowed with the $\sigma$-field $\mathcal F$ associated with the vague topology. We equip the canonical configuration space with the Poisson law
$\mathbb P$ of intensity $dt\,d\theta$. 
Throughout this work, for a sub-$\sigma$-field $\mathcal A\subseteq\mathcal F$, we denote by
$L^2(\Omega,\mathcal A,\mathbb P)$ the space of real-valued,
$\mathcal A$-measurable random variables $F$ satisfying
$\mathbb E[|F|^2]<\infty$, identified up to $\mathbb P$-almost sure
equality, and equipped with the norm
\[
\|F\|_{L^2(\Omega,\mathcal A,\mathbb P)}
:=
\bigl(\mathbb E[|F|^2]\bigr)^{1/2}.
\]
For a filtration $\mathbb G=(\mathcal G_t)_{t\in[0,T]}$, we denote
by $\mathcal P(\mathbb G)$ the predictable $\sigma$-field on
$[0,T]\times\Omega$, generated by all real-valued, left-continuous,
$\mathbb G$-adapted processes. Equivalently,
\[
\mathcal P(\mathbb G)
=
\sigma\!\left(
\bigl\{\{0\}\times A:A\in\mathcal G_0\bigr\}
\cup
\bigl\{(s,t]\times A:
0\le s<t\le T,\ A\in\mathcal G_s\bigr\}
\right).
\]
We also set
\[
\widehat{\mathcal P}(\mathbb G)
:=
\mathcal P(\mathbb G)\otimes\mathcal B(\mathbb R_+).
\]
On $(\Omega,\mathcal F,\P)$, let $\mathbf N$ denote the canonical Poisson random
measure,
\[
\mathbf N(\omega;A):=\omega(A),
\qquad
A\in\mathcal B(\mathbb X),
\]
with intensity measure $dt\,d\theta$. Thus, for every Borel set
$A\subset\mathbb X$ with finite Lebesgue measure, $\mathbf N(A)$ is Poisson
distributed with parameter $\int_A dt\,d\theta,$
and the random variables associated with disjoint sets are independent.  We denote by $\mathbb F^{\mathbf N}:=(\mathcal F_t^{\mathbf N})_{t\in [0,T]}$ the completed,
right-continuous augmentation of the natural filtration of $\mathbf N$, that is  for $t \in [0,T]$,
\[
\mathcal F_t^{\mathbf N}:=\sigma\left\{
\mathbf N((0,s]\times B):
0\leq s\leq t,\;
B\in\mathcal B(\mathbb R_+)\text{ bounded}
\right\},
\]
completed with the $\P$-negligible sets. We introduce the compensated Poisson random measure
\[
\widetilde{\mathbf N}(dt,d\theta)
:=
\mathbf N(dt,d\theta)-dt\,d\theta.
\]
%Note that a process $\lambda:[0,T]\times\Omega\to[0,\infty)$ is $\mathbb F^{\mathbf N}$-predictable if it is $\mathcal P(\mathbb F^{\mathbf N})$-measurable. 
We set $\lambda:[0,T]\times\Omega\to(0,\infty)$ 
a finite-valued process,  $\mathbb F^{\mathbf N}$-predictable that is measurable with respect to
the predictable $\sigma$-field $\mathcal{P}(\mathbb{F}^{\mathbf N})$. We define the simple counting process $C$ through the thinning procedure (also called Poisson embedding)
\begin{equation}\label{eq:thinning}
    C_t
    :=
    \int_{(0,t]\times\mathbb{R}_+}
    \mathbf{1}_{\{\theta\leq\lambda_s\}}\,\mathbf {N}(ds,d\theta),
    \qquad t\in [0,T],
\end{equation}
and let $\mathbb{F}^C$ denote the usual augmentation of
the natural filtration of $C$. Throughout the paper, the following assumption is in force.
\begin{assumption}\label{ass:lambda}
 $\lambda$ is
$\mathbb{F}^C$-predictable and satisfies \\
(i) $\sup\limits_{t\in[0,T]}
    \mathbb{E}[\lambda_t]<\infty$, \\
(ii)  there exists a deterministic constant $\mu>0$ such that:  \quad$
\lambda_t(\omega)\ge\mu,
\quad dt\otimes d\mathbb P\text{-a.e.}
$
\end{assumption}
 
\begin{remark}\label{rem:path} Since  $\lambda$  is $\mathbb F^C-$predictable, it  
can be specified by a measurable non-anticipative functional $\Lambda$:
\[
\lambda_t(\omega)
=
\Lambda\bigl(t,C(\omega)|_{[0,t)}\bigr).
\]
More precisely, if $\mathsf D$ denotes the space of locally finite
counting paths, the map
\[
(t,c)\longmapsto \Lambda\bigl(t,c|_{[0,t)}\bigr)\in(0,\infty)
\]
is measurable with respect to the predictable $\sigma$-field
of the canonical filtration on $\mathsf D$. We assume that the thinning equation
\[
C_t(\omega)
=
\int_{(0,t]\times\mathbb R_+}
\mathbf1_{\{
\theta\le
\Lambda(s,C(\omega)|_{[0,s)})
\}}
\,\omega(ds,d\theta)
\]
admits a measurable, non-explosive, pathwise unique solution. %An example of such process is the Hawkes process %defined below in Example \ref{example} given by the Br\'emaud-Massouli\'e (thinning) representation.
\end{remark}

\begin{example}[Linear Hawkes intensity]\label{example}
A particular example is the linear Hawkes process given by the Br\'emaud-Massouli\'e (thinning) representation, and whose intensity is given by 
\[
    \lambda_t
    =
    \mu+\int_{(0,t)}\Phi(t-s)\,dC_s,
\]
where $\mu>0$ and $\Phi:[0,T]\to\mathbb{R}_+$
is bounded  with
$\|\Phi\|_{L^1([0,T])}<1$, see \cite{ogata1981lewis,bremaud1996stability}.
\end{example}

In the thinning procedure $\eqref{eq:thinning}$, at each atom $(s,\theta)$ of $\mathbf N$, the point is retained precisely when
$\theta\leq\lambda_s$; the retained points are the jumps of $C$. 
Since $C$ is $\mathbb{F}^{\mathbf N}$-adapted, we have
$\mathcal{F}_t^C\subseteq\mathcal{F}_t^{\mathbf N}$ for every $t\in [0,T]$.
Moreover, under Assumption \ref{ass:lambda}
\[
    \mathbb{E}[C_T]
    =
    \mathbb{E}\!\left[\int_0^T\lambda_t\,dt\right]
    <\infty.
\]
Consequently, $C$ is non-explosive, and
\[
    M_t^C:=C_t-\int_0^t\lambda_s\,ds= \int_{(0,t]\times\R_+}
\mathbf 1_{\{\theta\leq\lambda_s\}}
\widetilde{\mathbf N}(ds,d\theta),
\qquad t\geq 0,
\]
is a square-integrable martingale in both
$\mathbb{F}^{\mathbf N}$ and $\mathbb{F}^C$.
Indeed, $M^C$ is $\mathbb F^C-$adapted since $\lambda$ is $\mathbb F^C-$predictable. Therefore,  using $\mathcal F_s^C\subseteq\mathcal F_s^{\mathbf N}$ and the tower property, we have for $0\leq s\leq t$
$$
\E\left[M_t^C\mid\mathcal F_s^C\right]
=
\E\left[
\E\left[M_t^C\mid\mathcal F_s^{\mathbf N}\right]
\,\middle|\,
\mathcal F_s^C
\right] =
\E\left[M_s^C\mid\mathcal F_s^C\right]
=
M_s^C.
$$ Hence, $M^C$ is also an $\mathbb F^C$-martingale.
As a consequence, the process
$\int_0^\cdot\lambda_s ds$ is the $\mathbb F^C$-compensator of $C$ and  $\lambda$ its  intensity.
Since $C$ is a simple counting process, the predictable quadratic
variation of $M^C$ in the filtration $\mathbb F^C$ is therefore
$$
\langle M^C\rangle_t
=
\int_0^t\lambda_s\,ds,
\qquad t\geq 0.
$$

\begin{definition}
For a filtration $\mathbb G$, let
\[
\mathbb L^2(\mathbf N;\mathbb G)
:=
\left\{
Z:[0,T]\times\Omega\times\mathbb R_+\to\mathbb R:
\begin{array}{l}
Z \text{ is }\widehat{\mathcal P}(\mathbb G)
\text{-measurable},\\[1mm]
\displaystyle
\mathbb E\!\left[
\int_0^T\int_{\mathbb R_+}|Z_{t,\theta}|^2\,d\theta\,dt
\right]<\infty
\end{array}
\right\},
\]
where processes are identified up to
$dt\otimes d\mathbb P\otimes d\theta$-almost everywhere equality. We also define
\[
L_C^2
:=
\left\{
U:[0,T]\times\Omega\to\mathbb R:
U \text{ is }\mathcal P(\mathbb F^C)\text{-measurable},
\quad
\mathbb E\!\left[\int_0^T |U_t|^2\lambda_t\,dt\right]<\infty
\right\},
\]
where processes are identified up to equality almost everywhere
with respect to the measure $\lambda_t(\omega)\,dt\,\mathbb P(d\omega)$.
These spaces are equipped with their respective norms
\[
\|Z\|_{\mathbb L^2(\mathbf N;\mathbb G)}^2
:=
\mathbb E\!\left[
\int_0^T\int_{\mathbb R_+}|Z_{t,\theta}|^2\,d\theta\,dt
\right],
\qquad
\|U\|_{L_C^2}^2
:=
\mathbb E\!\left[\int_0^T |U_t|^2\lambda_t\,dt\right].
\]
\end{definition}
Since $\lambda$ is $\mathbb F^C$-predictable and
$\mathbb F^C\subseteq\mathbb F^{\mathbf N}$, the map
\[
\iota:L_C^2\longrightarrow
\mathbb L^2(\mathbf N;\mathbb F^{\mathbf N}),
\qquad
(\iota U)_{t,\theta}:=U_t\mathbf 1_{\{\theta\le\lambda_t\}},
\]
is a linear isometry. Indeed,
\[
\|\iota U\|_{\mathbb L^2(\mathbf N;\mathbb F^{\mathbf N})}^2
=
\mathbb E\!\left[
\int_0^T\int_{\mathbb R_+}
|U_t|^2\mathbf 1_{\{\theta\le\lambda_t\}}\,d\theta\,dt
\right]
=
\|U\|_{L_C^2}^2.
\]
\begin{remark}
The notation $\mathbb L^2(\mathbf N;\mathbb G)$ refers to predictable
marked integrands, whereas $L^2(\Omega,\mathcal F_T^{\mathbf N},\mathbb P)$
denotes square-integrable terminal random variables.
\end{remark}

\subsection{Malliavin calculus for point processes}
\label{sec:MalliavinPointProcesses}
We introduce a Malliavin derivative adapted
to functionals of the point process $C$. The construction is inherited from
the Poisson embedding.

\begin{definition}[Shift operator]
\label{def:shift-operator}
For every $(t,\theta)\in\mathbb X$, we define the measurable map
\[
\begin{aligned}
\varepsilon_{t,\theta}^{+}:\Omega
&\longrightarrow\Omega,\\
\omega
&\longmapsto
\varepsilon_{t,\theta}^{+}(\omega),
\end{aligned}
\]
by requiring that, for every $A\in\mathcal B(\mathbb X)$,
\[
\bigl(\varepsilon_{t,\theta}^{+}(\omega)\bigr)(A)
:=
\omega\bigl(A\setminus\{(t,\theta)\}\bigr)
+
\mathbf 1_A(t,\theta).
\]
Thus, $\varepsilon_{t,\theta}^{+}(\omega)$ is the configuration obtained from
$\omega$ by adding a point at $(t,\theta)$.
\end{definition}
For every measurable functional $F:\Omega\to\mathbb R$, we denote by
$D_{t,\theta}$ the add-one derivative on the Poisson configuration space,
defined by
\[
D_{t,\theta}F
:=
F\circ\varepsilon_{t,\theta}^{+}-F.
\]

\begin{definition}[Malliavin--Sobolev space]
\label{def:HawkesSobolevSpace}
We define
\[
\mathbb D_{\mathbf N}^{1,2}
:=\left\{F\in L^2(\Omega,\mathcal F_T^{\mathbf N},\P):
\E\left[\int_0^T\int_{\mathbb R^+}|D_{t,\theta}F|^2\,d\theta dt\right]<+\infty\right\},
\]
endowed with the norm
\[\|F\|_{\mathbb D_{\mathbf N}^{1,2}}^2
:=
\mathbb E[|F|^2]
+
\mathbb E\!\left[
\int_0^T\int_{\mathbb R_+}|D_{t,\theta}F|^2\,d\theta\,dt
\right].\]
\end{definition}

\section{Martingale representation: from $\mathbb F^{\mathbf N}$ to $\mathbb F^C$ martingales}\label{sec:FN}

Now that the mathematical framework is well established, we will present results on the connections between the martingale representations in $\mathbb F^{\mathbf N}$ and $\mathbb F^C$, respectively. 

\subsection{Measurability characterization and martingale representation in $\mathbb F^C$}
We rely on the following key lemmas  which are consequences of the thinning representation to prove the  martingale representation theorem. It shows that although the auxiliary mark $\theta$ determines whether a Poisson point is
retained, its precise value is irrelevant once the point has been accepted.
\begin{lemma}[Reduction of the Poisson derivative on a smaller filtration]
\label{lem:hawkes-derivative}
Let $F$ be an $\mathcal F_T^C$-measurable real-valued random variable.
Then
\[
D_{t,\theta}F
=
\mathbf 1_{\{\theta\leq\lambda_t\}}D_{t,0}F,
\qquad
d\P\otimes dt\otimes d\theta\text{-a.e.}
\]
Equivalently,
\[
D_{t,\theta}F
=
D_{t,\bar\theta}F
\quad\text{on }
\{\theta\leq\lambda_t,\ \bar\theta\leq\lambda_t\},
\]
whereas
\[
D_{t,\theta}F=0
\quad\text{on }\{\theta>\lambda_t\}.
\]
\end{lemma}

\begin{proof}
For the accepted region $\{\theta\leq \lambda_t\}$, the result is a consequence of
\cite[Proposition~2.15 and Remark~2.18]{HHKR}.\footnote{Note that \cite[Proposition~2.15 and Remark~2.18]{HHKR} is related to the particular case of Hawkes processes. The proof can however be directly extended since it does not use the general specification of $\lambda$ and relies only on an approximation on the space of simple functional of $C,\lambda$. }
Since $\lambda_t\geq\mu>0$, the mark $0$ always belongs to the
accepted region and if $\theta\leq\lambda_t$, adding a point at $(t,\theta)$
produces the same perturbation of the $C$ process as adding the point
$(t,0)$. Hence,
\[
\mathbf 1_{\{\theta\leq\lambda_t\}}D_{t,\theta}F
=
\mathbf 1_{\{\theta\leq\lambda_t\}}D_{t,0}F.
\]
If $\theta>\lambda_t$, the inserted Poisson point is rejected by the thinning
procedure. Consequently, it does not create a jump of $H$ at time $t$.
Since the future intensity is determined by the past trajectory of $C$, the
insertion of this rejected point does not modify the subsequent evolution
of $C$. As $F$ is $\mathcal F_T^C$-measurable, it follows that
\[
F\circ\varepsilon_{t,\theta}^{+}=F
\qquad\text{on }\{\theta>\lambda_t\}.
\]
Therefore,
\[
D_{t,\theta}F=0
\qquad\text{on }\{\theta>\lambda_t\},
\]
which completes the proof.
\end{proof}

For later use, we recall the standard integration-by-parts formula on the
Poisson configuration space. For $F\in\mathbb D_{\mathbf N}^{1,2}$ and
$u\in\mathbb L^2(\mathbf N;\mathbb F^{\mathbf N})$,
the Poisson integration-by-parts formula (see  \cite{Privault_2009}) gives
\begin{equation}
\E\left[
F
\int_{(0,T]\times\mathbb R_+}
u_{t,\theta}\,
\widetilde{\mathbf N}(dt,d\theta)
\right]
=
\E\left[
\int_0^T\int_{\mathbb R_+}
D_{t,\theta}F\,
u_{t,\theta}\,d\theta\,dt
\right],
\label{eq:Poisson-IBP}
\end{equation}
A non-compensated  variant of this integration-by-parts formula is Mecke's formula,  that takes the following form:
\begin{equation}\label{eq:Mecke}
 \E\left[F \int_{(0,T]\times \real_+}  u_{t,\theta} {\mathbf N} (dt,d\theta)\right] = \E\left[\int_0^T \int_{\real_+}  u_{t,\theta}  \, (F \circ \eps_{(t,\theta)}^+) d\theta \, dt \right].
\end{equation}

Combining \eqref{eq:Poisson-IBP} (resp. \eqref{eq:Mecke}) with
Lemma~\ref{lem:hawkes-derivative} immediately yields an integration-by-parts
formula with respect to the compensated point process  $C$  (resp. non compensated  point process $\tilde C$).

\begin{corollary}[Integration by parts]
\label{cor:preliminary-Hawkes-IBP}
Let $F\in L^2(\Omega,\mathcal F_T^C,\P)$
be such that
\[
\E\left[
\int_0^T
|D_{t,0}F|^2\lambda_t\,dt
\right]
<+\infty.
\]
Then, for every deterministic $h\in L^2([0,T],dt)$,
\[
\E\left[
F\int_0^T h_t\,dM_t^C
\right]
=
\E\left[
\int_0^T
D_{t,0}F\,h_t\lambda_t\,dt
\right].
\]
\end{corollary}

\begin{proof}
By the thinning representation,
\[
\int_0^T h_t\,dM_t^C
=
\int_{(0,T]\times\mathbb R_+}
h_t\mathbf 1_{\{\theta\leq\lambda_t\}}
\,\widetilde{\mathbf N}(dt,d\theta).
\]
By Lemma~\ref{lem:hawkes-derivative} and the assumption on $F$,

$$
\E\left[
\int_0^T\int_{\R_+}
|D_{t,\theta}F|^2\,d\theta\,dt
\right]
=
\E\left[
\int_0^T
|D_{t,0}F|^2\lambda_t\,dt
\right]
<+\infty.
$$
Hence $F\in\mathbb D_{\mathbf N}^{1,2}$, and the Poisson integration-by-parts formula~\eqref{eq:Poisson-IBP} applies:
\begin{align*}
\E\left[
F\int_0^T h_t\,dM_t^C
\right]
&=
\E\left[
\int_0^T\int_{\mathbb R_+}
D_{t,\theta}F\,
h_t\mathbf 1_{\{\theta\leq\lambda_t\}}
\,d\theta\,dt
\right].
\end{align*}
By Lemma~\ref{lem:hawkes-derivative},
\[
D_{t,\theta}F
=
D_{t,0}F
\qquad\text{on }\{\theta\leq\lambda_t\}.
\]
Therefore,
$$
\E\left[
F\int_0^T h_t\,dM_t^C
\right]
=
\E\left[
\int_0^T
\int_0^{\lambda_t}
D_{t,0}F\,h_t\,d\theta\,dt
\right]
=
\E\left[
\int_0^T
D_{t,0}F\,h_t\lambda_t\,dt
\right].
$$
\end{proof}
 We define now   \[
\mathbb D_{C}^{1,2}
:=\left\{F\in L^2(\Omega,\mathcal F_T^{C},\P):
\E\left[\int_0^T|D_{t,0}F|^2\lambda_t\,dt\right]<+\infty\right\},
\]
equipped with the norm
\[
\|F\|_{\mathbb D_C^{1,2}}^2
:=
\mathbb E[|F|^2]
+
\mathbb E\!\left[
\int_0^T |D_{t,0}F|^2\lambda_t\,dt
\right].
\]

\begin{corollary} For any random variable $F\in \mathbb D_C^{1,2}\subset \mathbb D_{\mathbf N}^{1,2}$, we have
 $\|F\|_{\mathbb D_C^{1,2}}
=
\|F\|_{\mathbb D_{\mathbf N}^{1,2}}.$
 \end{corollary}

\begin{proof}
    Let $F\in \mathbb D^{1,2}_C$. By using Lemma \ref{lem:hawkes-derivative} we note that

   \[ \E\left[\int_0^T\int_{\mathbb R^+}|D_{t,\theta}F|^2\,d\theta dt\right]= \E\left[\int_0^T|D_{t,0}F|^2 \lambda_t dt\right]<\infty,\]
   since $F\in L^2(\Omega,\mathcal F_T^{C},\P)$.
\end{proof}
We now want to develop the converse part of Lemma \ref{lem:hawkes-derivative}. The following lemma is the key result to derive the representation martingale theorem in the $\mathbb F^C-$filtration. Note that this result has not been established in preliminaries studies as for example \cite{HHKR,hillairet2024poisson}, while it is fundamental for our study. 
\begin{lemma}\label{lemma:derivetoH}
Let $F \in L^2(\Omega, \mathcal{F}_T^{\mathbf N}, \mathbb P)$. If
\begin{equation}\label{formuleDFmark}
D_{t,\theta}F = 1_{\{\theta \leq \lambda_t\}} D_{t,0}F,\; d\mathbb P \otimes dt \otimes d\theta-a.e.,
\end{equation}
then $F$ is $\mathcal{F}_T^C$-measurable.
\end{lemma}

\begin{proof}
 Set $\widetilde F:=\mathbb E[F\mid\mathcal F_T^C]$, and $
\overline F:=F-\widetilde F.$
Then, $\mathbb E[\overline F\mid\mathcal F_T^C]=0.$
By Lemma~\ref{lem:hawkes-derivative} applied to $\widetilde F$ and the assumption on $F$, we deduce by linearity of the Malliavin derivative
\begin{equation}
D_{t,\theta}\overline F
=
\mathbf 1_{\{\theta\leq\lambda_t\}}D_{t,0}\overline F,
\qquad
d\mathbb P\otimes dt\otimes d\theta\text{-a.e.}
\label{eq:lemma34-residual}
\end{equation}
For every bounded nonnegative measurable function $f$ vanishing
outside $(0,T]\times[0,K]$ for some $K<\infty$, define
\[
I(f):=\int_{(0,T]\times\mathbb R_+}
f(t,\theta)\,\mathbf N(dt,d\theta).
\]
The linear span of the random variables $e^{-I(f)}$, as $f$
ranges over this class, is dense in
$L^2(\Omega,\mathcal F_T^{\mathbf N},\mathbb P)$,
see \cite[Lemma~18.4]{last2017penrose}.
To prove that $\overline F=0$ it  therefore suffices to prove
\begin{equation}
\mathbb E[\overline F e^{-I(f)}]=0
\label{eq:lemma34-target}
\end{equation}
for any such $f$.
To achieve \eqref{eq:lemma34-target} we divide the proof in several steps.

\textit{Step 1: Assume that $F$ is bounded.} So that $\overline F$ is bounded. Let $(T_i,\Theta_i)_{1\leq i\leq H_T}$ denote the retained atoms,
ordered by time. We start by decomposing $e^{I(f)}$ with jumps accepted by the thinning and those which are rejected, see for example,
\cite{hillairet2024chaotic}:
\[
I(f)
=
\sum_{i=1}^{C_T}f(T_i,\Theta_i)
+
\int_{(0,T]\times\mathbb R_+}
\mathbf 1_{\{\theta>\lambda_t\}}f(t,\theta)\,
\mathbf N(dt,d\theta).
\]
Consequently,

$$e^{-I(f)}=  \exp\left(
-\int_{(0,T] \times\mathbb R_+}
\mathbf 1_{\{\theta>\lambda_s\}}f(s,\theta)\,\mathbf N(ds,d\theta)\right) \prod_{i=1}^{C_T} a(T_i,\theta_i),\; a(t,\theta):=e^{-f(t,\theta)}.$$
We compensate the exponential on the rejected domain to get an exponential martingale and define
\[
Y_t:=
\exp\left(
-\int_{(0,t]\times\mathbb R_+}
\mathbf 1_{\{\theta>\lambda_s\}}f(s,\theta)\,\mathbf N(ds,d\theta)
+
\int_0^t\int_{\lambda_s}^{\infty}
\bigl(1-a(s,\theta)\bigr)\,d\theta\,ds
\right),
\] and we set $B$ this compensator at time $T$:
\[
B:=
\exp\left(
-\int_0^T\int_{\lambda_t}^{\infty}
\bigl(1-a(t,\theta)\bigr)\,d\theta\,dt
\right).
\] In particular, $B$ is $\mathcal F_T^C$-measurable. From It\^o-jump formula applied to $Y$ we note that
\begin{equation}
Y_T-1
=
\int_{(0,T]\times\mathbb R_+}
Y_{t-}\bigl(a(t,\theta)-1\bigr)
\mathbf 1_{\{\theta>\lambda_t\}}
\,\widetilde {\mathbf N}(dt,d\theta).
\label{eq:lemma34-Y-integral}
\end{equation}
The integrand is $\mathbb F^{\mathbf N}-$predictable and square-integrable. Indeed, since $f\geq 0$ and $f(t,\theta)=0$ for $\theta>K$,
\[
0\leq 1-a(t,\theta)
=1-e^{-f(t,\theta)}
\leq \mathbf 1_{\{\theta\leq K\}}.
\]
The integral with respect to $\mathbf N$ in the exponent defining
$Y_t$ is nonpositive. Therefore,
\[
0<Y_t
\leq
\exp\left(
\int_0^t\int_{\lambda_s}^{\infty}
\bigl(1-a(s,\theta)\bigr)\,d\theta\,ds
\right)
\leq
\exp\left(\int_0^t(K-\lambda_s)_+\,ds\right)
\leq e^{KT},
\quad t\in[0,T].
\]
Consequently,
\[
\mathbb E\left[
\int_0^T\int_{\mathbb R_+}
\left|
Y_{t-}\bigl(a(t,\theta)-1\bigr)
\mathbf 1_{\{\theta>\lambda_t\}}
\right|^2\,d\theta\,dt
\right]
\leq KT e^{2KT}<\infty.
\]
By construction,
\begin{equation}
e^{-I(f)}=ABY_T,  \quad \text{ where } \; A:=\prod\limits_{i=1}^{C_T}a(T_i,\Theta_i)
\label{eq:lemma34-factorization},
\end{equation}
with the convention that empty products equal $1$. Note that $A$ and $B$
take values in $(0,1]$.

\medskip
\noindent
\textit{Step 1a: Elimination of the rejected-point exponential.}

Set $G:=\overline F AB$. Since $A,B$ are both $\mathcal F_T^C$-measurable, they are unchanged by inserting a rejected point.
Together with \eqref{eq:lemma34-residual}, this implies
\[
D_{t,\theta}G=0
\quad\text{on }\{\theta>\lambda_t\},
\qquad
d\mathbb P\otimes dt\otimes d\theta\text{-a.e.}
\]
Since $G$ is bounded, by Cauchy-Schwarz inequality, we have
\[
\mathbb E\left[
\int_0^T\int_{\mathbb R_+}
|D_{t,\theta}G|^2\,d\theta\,dt
\right]
\leq
4\|\overline F\|_\infty^2
\mathbb E\left[\int_0^T\lambda_t\,dt\right]
<\infty.
\]
The last integrability follows from the standing assumption on $\lambda$. Thus $G\in\mathbb D_{\mathbf N}^{1,2}$.
Applying the Poisson integration-by-parts formula to $G$ and
the integrand in \eqref{eq:lemma34-Y-integral} yields
\[
\begin{aligned}
\mathbb E[G(Y_T-1)]
=
\mathbb E\left[
\int_0^T\int_{\mathbb R_+}
D_{t,\theta}G\,
Y_{t-}\bigl(a(t,\theta)-1\bigr)
\mathbf 1_{\{\theta>\lambda_t\}}
\,d\theta\,dt
\right] =0.
\end{aligned}
\]
Hence, by \eqref{eq:lemma34-factorization},
\begin{equation}
\mathbb E[\overline F e^{-I(f)}]
=
\mathbb E[\overline F AB].
\label{eq:lemma34-remove-Y}
\end{equation}

\medskip
\noindent
\textit{Step 1b: Averaging the retained marks.}\\
Define 
\[g_t:=\frac{1}{\lambda_t}\int_0^{\lambda_t}a(t,\theta)\,d\theta,\;  \quad \mbox{and} \quad
\widetilde A:=\prod_{i=1}^{C_T}g_{T_i}.\]
The telescoping product identity gives
\begin{equation}
A-\widetilde A
=
\sum_{k=1}^{C_T}
\left(\prod_{i<k}g_{T_i}\right)
\bigl(a(T_k,\Theta_k)-g_{T_k}\bigr)
\left(\prod_{i>k}a(T_i,\Theta_i)\right).
\label{eq:lemma34-telescoping}
\end{equation}
Our goal is to show that
\begin{equation}
\mathbb E[\overline F B(A-\widetilde A)]=0.
\label{eq:lemma34-mark-averaging}
\end{equation}
For a configuration $\omega$, write
\[
P_t(\omega):=
\prod_{T_i(\omega)<t}g_{T_i(\omega)}(\omega),
\qquad
Q_t(\omega):=
\prod_{T_i(\omega)>t}
a(T_i(\omega),\Theta_i(\omega)).
\]
Then \eqref{eq:lemma34-telescoping} becomes
\[
A-\widetilde A
=
\int_{(0,T]\times\mathbb R_+}
\mathbf 1_{\{\theta\leq\lambda_t\}}
P_tQ_t\bigl(a(t,\theta)-g_t\bigr)\,{\mathbf N}(dt,d\theta).
\]
Write $\omega^{t,\theta}:=\omega+\delta_{(t,\theta)}$.
Insertion at time $t$ leaves the history strictly before $t$
unchanged. In particular,
\[
\lambda_t(\omega^{t,\theta})=\lambda_t(\omega),
\qquad
g_t(\omega^{t,\theta})=g_t(\omega),
\qquad
P_t(\omega^{t,\theta})=P_t(\omega).
\]
Mecke's formula \eqref{eq:Mecke} therefore gives 
\begin{equation}
\begin{aligned}
&\mathbb E[\overline F B(A-\widetilde A)]\\
&\quad=
\mathbb E\left[
\int_0^T\int_0^{\lambda_t(\omega)}
\overline F(\omega^{t,\theta})
B(\omega^{t,\theta})
P_t(\omega)
Q_t(\omega)
\bigl(a(t,\theta)-g_t(\omega)\bigr)
\,d\theta\,dt
\right].
\end{aligned}
\label{eq:lemma34-Mecke}
\end{equation}
This is justified by the boundedness of the factors, especially $\overline F$ for this case that we will relax in the last step of the proof and
\[
\mathbb E[C_T]
=
\mathbb E\left[\int_0^T\lambda_t\,dt\right]
<\infty.
\]
Recall that \eqref{eq:lemma34-residual} implies that
\[
\overline F(\omega^{t,\theta})
=
\overline F(\omega)+D_{t,0}\overline F(\omega)
\]
throughout the accepted region, up to the stated null sets.
Furthermore, inserting a point at time $t$ with any accepted mark
produces the same trajectory than $C$. Hence
$B(\omega^{t,\theta})$ is independent of the accepted inserted mark and for any $\theta,\theta'\leq \lambda_t(\omega)$, the thinning
construction yields
\[
C(\omega^{t,\theta})=C(\omega^{t,\theta'}),
\qquad
B(\omega^{t,\theta})=B(\omega^{t,\theta'}).
\]
 Consequently, for $d\mathbb P\otimes dt$-almost every $(\omega,t)$, the factor
\[
\overline F(\omega^{t,\theta})
B(\omega^{t,\theta})
P_t(\omega)
Q_t(\omega)
\]
is independent of $\theta$, for Lebesgue-almost every
$\theta\in(0,\lambda_t(\omega)]$. Since 

\[
\int_0^{\lambda_t(\omega)}
\bigl(a(t,\theta)-g_t(\omega)\bigr)\,d\theta=0, 
\]the inner integral in \eqref{eq:lemma34-Mecke}
vanishes. This proves \eqref{eq:lemma34-mark-averaging}. \\

\textit{Step 1c. Orthogonality and conclusion for the bounded case.} \\
Combining \eqref{eq:lemma34-remove-Y} and
\eqref{eq:lemma34-mark-averaging}, we obtain
\[
\begin{aligned}
\mathbb E[\overline F e^{-I(f)}]
=\mathbb E[\overline F BA]=\mathbb E[\overline F B\widetilde A]
=\mathbb E\left[
B\widetilde A\,
\mathbb E[\overline F\mid\mathcal F_T^C]
\right]
=0.
\end{aligned}
\]
By totality of the Poisson exponentials, $\overline F=0$
almost surely. Thus $F$ is $\mathcal F_T^C$-measurable in the
bounded case.

\medskip
\noindent
\textit{Step 2: Extension to $F\in L^2(\Omega,\mathcal F_T^{\mathbf N},\mathbb P)$.} For $n\geq1$, define
\[
\varphi_n(x):=(-n)\vee(x\wedge n),
\qquad
F_n:=\varphi_n(F).
\]
The derivative condition is preserved by truncation. Indeed,
\[
\begin{aligned}
D_{t,\theta}F_n
&=
\varphi_n(F+D_{t,\theta}F)-\varphi_n(F)\\
&=
\mathbf 1_{\{\theta\leq\lambda_t\}}
\bigl[\varphi_n(F+D_{t,0}F)-\varphi_n(F)\bigr]\\
&=
\mathbf 1_{\{\theta\leq\lambda_t\}}D_{t,0}F_n,
\end{aligned}
\]
for $d\mathbb P\otimes dt\otimes d\theta$-almost every
$(\omega,t,\theta)$.
The bounded case shows that each $F_n$ is
$\mathcal F_T^C$-measurable.
Since $F_n\to F$ in $L^2(\Omega,\mathcal F_T^{\mathbf N},\mathbb P)$, which is a closed space of variables, the conclusion follows.
\end{proof}

As a direct consequence of Lemma \ref{lem:hawkes-derivative} and Lemma \ref{lemma:derivetoH} we have the following characterization of the $\mathbb F^C-$ measurability.

\begin{corollary}\label{cor:FH}
    Let $F \in L^2(\Omega, \mathcal{F}_T^{\mathbf N}, \mathbb P)$. Then 
\[
D_{t,\theta}F = 1_{\{\theta \leq \lambda_t\}} D_{t,0}F,\; d\mathbb P \otimes dt \otimes d\theta-a.e.,
\text{ if and only if } F \text{ is $\mathcal{F}_T^C$-measurable.}\]
\end{corollary}

Our goal is to combine this characterization with the
martingale representation theorem for $\mathbf N$ in order to deduce the
martingale representation property in the filtration $\mathbb F^C$ directly. 
We first recall the predictable representation property of the baseline
Poisson random measure; see, for instance, \cite{Privault_2009,Last2016}.

\begin{theorem}
\label{th:PRPgrosPoisson}
Let $M:=(M_t)_{t\in[0,T]}$ be a square-integrable
$\mathbb F^{\mathbf N}$-martingale. Then there exists a unique
$Z\in\mathbb L^2(\mathbf N;\mathbb F^{\mathbf N})$ such that
\[
M_t=M_0+\int_{(0,t]\times\mathbb R_+}
Z_{s,\theta}\,\widetilde{\mathbf N}(ds,d\theta),
\qquad t\in[0,T],\quad\P\text{-a.s.}
\]
%with 
and the integrand is $Z_{t,\theta}=\mathbb{E}\big[D_{t,\theta}M_T\,\big|\,\mathcal{F}^{\mathbf N}_{t-}\big],\qquad d\mathbb{P}\otimes dt\otimes d\theta\text{-a.e.}$
\end{theorem}

\begin{remark}
\label{rem:counterexample-N-martingale}
The preceding theorem does not guarantee that the integrand is $\mathbb F^C-$predictable for an arbitrary
$\mathbb F^{\mathbf N}$-martingale. Indeed, consider
\[
E_t
:=
\int_{(0,t]\times\mathbb R_+}
\mathbf 1_{\{\lambda_s<\theta\leq\lambda_s+1\}}
\,\widetilde{\mathbf N}(ds,d\theta),
\qquad t\in[0,T].
\]
Since the integrand is $\mathbb F^{\mathbf N}-$predictable and $\int_{\mathbb R_+}
\mathbf 1_{\{\lambda_s<\theta\leq\lambda_s+1\}}
\,d\theta=1,$ the process $E$ is a square-integrable
$\mathbb F^{\mathbf N}$-martingale and
$
\E[E_t^2]=t.$
However, $E$ only jumps at Poisson points $(s,\theta)$ satisfying
$\theta>\lambda_s$. Such points are rejected by the thinning
procedure and therefore do not produce jumps of $C$. Consequently,
$E$ is not adapted to the natural filtration $\mathbb F^C$. In particular, $E$ cannot admit a representation of the form
$
E_t
=
E_0+\int_0^t U_s\bigl(dC_s-\lambda_s\,ds\bigr)$
for some $\mathbb F^C$-predictable process $U$, since the right-hand
side can jump only at jump times of $C$, whereas \(E\) jumps only at
rejected Poisson points.
\end{remark}
Our objective is to establish the martingale representation form Theorem \ref{th:PRPgrosPoisson} for any square-integrable $\mathbb F^C-$martingale $M=(M_t)_{t\geq 0}$, without assuming that $M$ is also an $\mathbb F^{\mathbf N}$-martingale with an $\mathbb F^{\mathbf N}$ predictable integrand. To do so, we  use Malliavin calculus to characterize the integrand in the representation with respect to the compensated process $M^C=C-\int_0^. \lambda_s ds$.
The main result of this work  is the following theorem.

\begin{theorem}[Martingale representation property]
\label{thm:hawkes-mrp}

Let $P=(P_t)_{t\in[0,T]}$ be a square-integrable
$\mathbb F^C$-martingale. Then there exists a unique
$U\in L_C^2$ such that
\[
P_t=P_0+\int_0^t U_s\bigl(dC_s-\lambda_s\,ds\bigr),
\qquad t\in[0,T],\quad \mathbb P\text{-a.s..}
\]
Equivalently, 
$$
P_t
=
P_0+
\int_{(0,t]\times\mathbb R_+}
U_s\mathbf 1_{\{\theta\leq\lambda_s\}}
\,\widetilde{\mathbf N}(ds,d\theta).
$$
In particular, $P$ is also an  $\mathbb F^{\mathbf N}$-martingale and the integrand in the martingale representation with respect to  $\widetilde{\mathbf N}$ (see Theorem \ref{th:PRPgrosPoisson}) is given by %je pense qu'il faut insister sur le fait que c'est une  aussi une $\mathbb F^N$-martingale-  je mettrai:  and  
\[
Z_{s,\theta} =\E\left[
D_{s,\theta}P_T
|
\mathcal F_{s-}^{\mathbf N}
\right]=U_s\mathbf 1_{{\theta\leq\lambda_s}},
\quad
d\P\otimes ds\otimes d\theta\text{-a.e.}
.\]
%where $Z_{s,\theta}=\E\left[ D_{s,\theta}P_T |\mathcal F_{s-}^{\mathbf N} \right]$ is the integrand in martingale representation Theorem \ref{th:PRPgrosPoisson}.
\end{theorem}
\begin{proof}
Let $P=(P_t)_{t\in[0,T]}$ be a square-integrable
$\mathbb F^C$-martingale and define the
$\mathbb F^{\mathbf N}$-martingale $Q=(Q_t)_{t\in[0,T]}$ by

$$
Q_t
:=
\E\left[P_T|\mathcal F_t^{\mathbf N}\right],
\qquad t\in[0,T].
$$
By Theorem~\ref{th:PRPgrosPoisson}, there exists a unique
$Z\in \mathbb L^2(\mathbf N;\mathbb F^{\mathbf N})$ such that
\[
Q_t=Q_0+
\int_{(0,t]\times\mathbb R_+}
Z_{s,\theta}
\widetilde{\mathbf N}(ds,d\theta),
\qquad t\in[0,T].
\]
Moreover, $
Q_0=\E[P_T]=P_0.
$ By the
Clark-Ocone identification,

$$
Z_{s,\theta}
=
\E\left[
D_{s,\theta}P_T
|
\mathcal F_{s-}^{\mathbf N}
\right].
$$
Since $P_T$ is $\mathcal F_T^C$-measurable, Lemma~\ref{lem:hawkes-derivative} gives

$$
D_{s,\theta}P_T
=
\mathbf 1_{\{\theta\leq\lambda_s\}}D_{s,0}P_T.
$$
Since $\lambda_s$ is $\mathbb F^C$-predictable, it is in
particular $\mathcal F_{s-}^{\mathbf N}$-measurable. Hence
\[
Z_{s,\theta}=
\E\left[
\mathbf 1_{{\theta\leq\lambda_s}}D_{s,0}P_T|
\mathcal F_{s-}^{\mathbf N}
\right]=
\mathbf 1_{{\theta\leq\lambda_s}}
\E\left[
D_{s,0}P_T
|
\mathcal F_{s-}^{\mathbf N}
\right].
\]
We then define
\[
U_s
:=
\E\left[
D_{s,0}P_T
\mid
\mathcal F_{s-}^{\mathbf N}
\right].
\]
Thus 
\[
Z_{s,\theta}=U_s\mathbf 1_{{\theta\leq\lambda_s}},
\quad
d\P\otimes ds\otimes d\theta\text{-a.e.}
\]
Equivalently, since \(\lambda_s\geq\mu>0\),

$$
U_s
=
\frac1{\lambda_s}
\int_0^{\lambda_s}Z_{s,\theta}\,d\theta.
$$
Note that the process $U$ satisfies \eqref{formuleDFmark} in Lemma \ref{lemma:derivetoH} and so $U$ admits an
$\mathbb F^C$-predictable version. We keep the notation $U$ for
this version. Hence,

\[ Q_t=Q_0+
\int_{(0,t]\times\mathbb R_+}
Z_{s,\theta}
\widetilde{\mathbf N}(ds,d\theta)=Q_0 
+\int_{(0,t]\times\mathbb R_+}
U_s(dC_s-\lambda_s ds).
\]
Consequently, $Q_t$ is $\mathcal F^C_t-$measurable and since $
\mathcal F_t^C\subseteq\mathcal F_t^{\mathbf N},
$ the tower property gives

$$
\begin{aligned}
P_t
&=
\E[P_T\mid\mathcal F_t^C]
\\
&=
\E\left[
\E[P_T\mid\mathcal F_t^{\mathbf N}]
\mid\mathcal F_t^C
\right]
\\
&=
\E[Q_t\mid\mathcal F_t^C]\\
&=Q_t\\
&=
P_0+
\int_0^tU_s\bigl(dC_s-\lambda_s\,ds\bigr).
\end{aligned}
$$
Using the thinning identity once again, this is equivalently

\[
P_t
=
P_0+
\int_{(0,t]\times [0,K]}
U_s\mathbf 1_{\{\theta\leq\lambda_s\}}
\,\widetilde{\mathbf N}(ds,d\theta).
\]
It remains to verify the square-integrability of \(U\):
\[
\|U\|_{L_C^2}^2
=
\mathbb E\!\left[
\int_0^T\int_{\mathbb R_+}
|U_s|^2\mathbf 1_{\{\theta\le\lambda_s\}}\,d\theta\,ds
\right]
=
\|Z\|_{\mathbb L^2(\mathbf N;\mathbb F^{\mathbf N})}^2
<\infty.
\]

\end{proof}
\begin{remark}
In particular, 
    $$
U_s
=
\frac1{\lambda_s}
\int_0^{\lambda_s}Z_{s,\theta}\,d\theta.
$$
\end{remark}

%\begin{remark}
%Note from \eqref{PrepN} that $P$ is also a square-integrable \(\mathbb F^{\mathbf N}\)-martingale. \textcolor{blue}{explain intuition with enlargement of filtration}
%\tr{ Note from \eqref{PrepN}, which uses the thinning identity, $\mathbb F^C$-martingales remain martingales in the larger filtration  \(\mathbb F^{\mathbf N}\). 
%la je dois travailler un peu plus  pour verifier hyp de Jacod de enlargement of filtration car on n'a pas completement independance de la marque puisque  le support de sa loi uniforme depend de $\lambda_t$ qui est $\mathbb F^C$-previsible. }
%\end{remark}

\begin{remark}[Connection with filtration enlargement and immersion]
The inclusion $\mathbb{F}^{C}\subseteq\mathbb{F}^{\mathbf N}$
can be viewed as an enlargement of the observed filtration by
revealing the auxiliary marks and rejected Poisson points as
time evolves. In general, enlarging a filtration does not preserve
the martingale property: the additional information may change
the compensator and introduce a drift correction. For marked
point processes, \cite{GapeevJeanblancWu2025}
studies initial and progressive enlargements under Jacod's absolute
continuity hypothesis and derive explicit formulas for martingale
representations and their optional projections. 
Our framework is more complex as it is neither strictly initial nor progressive enlargement . In a sense, it combines successive progressive enlargements of filtration, driven by the rejected times, with successive initial enlargements,  incorporating the marks across all jump times (both retained and rejected).
Our approach, based on  the  thinning construction, yields in our setting the stronger property that
$\mathbb{F}^{C}$ is immersed in $\mathbb{F}^{\mathbf N}$ under
$\mathbb{P}$, also called $(H)$ Hypothesis. Indeed, Theorem~\ref{thm:hawkes-mrp}
gives, for every square-integrable $\mathbb{F}^{C}$-martingale $P$,
\[
P_t
=
P_0+\int_{(0,t]\times\mathbb{R}_+}
U_s\mathbf{1}_{\{\theta\leq\lambda_s\}}
\,\widetilde{\mathbf N}(ds,d\theta),
\]
where $U\in L_C^2$. The integrand is
$\mathbb{F}^{\mathbf N}$-predictable and square-integrable,
so $P$ remains a martingale in the larger filtration.
Equivalently, for every
$X\in L^2(\Omega,\mathcal F_T^C,\mathbb P)$,
\[
\mathbb E[X\mid\mathcal F_t^{\mathbf N}]
=
\mathbb E[X\mid\mathcal F_t^C],
\qquad t\in[0,T].
\]
Thus the additional Poisson information available at time $t$
does not improve the prediction of functionals of $C$ beyond its
observed history. This property follows from the fact that $\lambda$ is the intensity
of $C$ in both filtrations. Consequently, no separate
verification of Jacod's density hypothesis is needed for this
particular enlargement. The advantage of our approach is to
establish this compatibility directly on a fixed Poisson
reference space while identifying the representation integrand
in the smaller filtration through Malliavin calculus, whenever
the terminal variable has the required regularity. 
\end{remark}

\subsection{Application to general Hawkes process}\label{sec:application}

As an application, extending the linear Hawkes process in Example \ref{example} to a more general case, we derive a martingale representation theorem for a general Hawkes process from the thinning representation. Let $C=H$ be the Hawkes
process with empty initial history, constructed from the
Poisson random measure $\mathbf N$ through
\[
    H_t
    =
    \int_{(0,t]\times\mathbb R_+}
        \mathbf 1_{\{\theta\leq\lambda_s\}}
        \,\mathbf N(ds,d\theta),
    \qquad
    \lambda_t
    =
    \mu+\Psi\Big(\int_{(0,t)}\Phi(t-s)\,dH_s\Big),
\]
where $\mu>0$, $\Psi:\mathbb R\longrightarrow \mathbb R_+$ is Lipschitz and $\Phi:[0,T]\to\mathbb R$ is 
bounded with
\[
    \|\Psi\|_{\mathrm{Lip}}\times\|\Phi\|_{L^1([0,T])}<1,
\]
for the Lipschitz norm defined by $\|\Psi\|_{\mathrm{Lip}}
:=
\sup_{\substack{x,y\in \mathbb R\\x\neq y}}
\frac{|\Psi(x)-\Psi(y)|}{|x-y|}.$

\begin{corollary}[Martingale representation for Hawkes processes]
\label{cor:hawkes-mrt}
Every square-integrable $\mathbb F^H$-martingale
$P=(P_t)_{t\in[0,T]}$ admits the representation
\[
    P_t
    =
    P_0+\int_0^t U_s\,dM_s^H
    =
    P_0+\int_0^t U_s(dH_s-\lambda_s\,ds),
    \qquad t\in[0,T],
\]
up to indistinguishability, for a unique $U\in L_H^2$ with

\[
    \mathbb E[|P_T-P_0|^2]
    =
    \mathbb E\!\left[
        \int_0^T |U_s|^2\lambda_s\,ds
    \right].
\]
In particular, for every $F\in L^2(\Omega,\mathcal F_T^H,\mathbb P)$,
there is a unique $U\in L_H^2$ such that for any $t\in [0,T]$

\[
    \mathbb E[F\mid\mathcal F_t^{\mathbf N}]
    =
    \mathbb E[F\mid\mathcal F_t^H]
    =
    \mathbb E[F]+\int_0^t U_s\,dM_s^H.
\]
If, in addition, $F\in\mathbb D_{\mathbf N}^{1,2}$,
then the integrand is identified by
\[
    U_s
    =
    \mathbb E[D_{s,0}F\mid\mathcal F_{s-}^{\mathbf N}]
    =
    \mathbb E[D_{s,0}F\mid\mathcal F_{s-}^{H}],
    \qquad d\mathbb P\otimes ds\text{-a.e.},
\]
where the conditional expectations are understood in
their predictable versions.
\end{corollary}
\begin{proof}
The Hawkes intensity is $\mathbb F^H$-predictable and
satisfies $\lambda_t\geq\mu>0$. Since $\Phi$ is bounded,
the thinning equation admits a pathwise unique,
non-explosive solution, see \cite[Theorem 1(a)]{bremaud1996stability}
\[
    \sup_{t\in[0,T]}\mathbb E[\lambda_t]<\infty.
\]
Hence Assumption~\ref{ass:lambda} and the construction requirements
of Remark~\ref{rem:path} are satisfied. The result follows
directly from Theorem~\ref{thm:hawkes-mrp} with $C=H$. Applying the theorem to
$P_t=\mathbb E[F|\mathcal F_t^H]$ and using the
$\mathbb P$-triviality of $\mathcal F_0^H$, which gives
$P_0=\mathbb E[F]$.
\end{proof}

\section{ Orthogonal martingale decomposition in the joint accepted-rejected filtration}\label{sec:ortho}
Finally, we show that the Poisson representation of Theorem \ref{th:PRPgrosPoisson} splits along the two regions of the thinning. To avoid degeneracy of the rejected region we assume, in addition to Assumption \ref{ass:lambda}, that there exists a constant $\bar\mu>0$ such that $\lambda_s\le\bar\mu$, $ds\otimes d\mathbb{P}$-a.e., and we define
\[
\bar C_t:=\int_{(0,t]\times[0,\bar\mu]}\mathbf{1}_{\{\theta>\lambda_s\}}\,\mathbf N(ds,d\theta),\text{ with }
\bar\lambda_s:=\bar\mu-\lambda_s,\text{ and }
M^{\bar C}_t:=\bar C_t-\int_0^t\bar\lambda_s\,ds,\qquad t\in[0,T].
\]
Since the integrand $\mathbf{1}_{[0,\bar\mu]}(\theta)\mathbf{1}_{\{\theta>\lambda_s\}}$ is bounded and $\mathbb{F}^{\mathbf N}$-predictable, $M^{\bar C}_t$ is a square-integrable $\mathbb{F}^{\mathbf N}$-martingale, that is, $\bar\lambda$ is the $\mathbb{F}^{\mathbf N}$-intensity of $\bar C$. The filtration %in which both processes $C,\bar C$ are observed is denoted by 
$\mathbb{K}=(\mathcal{K}_t)_{t\in[0,T]}$ denotes the usual augmentation of the filtration generated by $(C,\bar C)$, \textit{i.e.}\ the internal history of the two-type point process recording the accepted atoms and the rejected atoms with mark in $[0,\bar\mu]$. %, with $\mathcal{H}_0$ trivial. 
Then $\mathbb{F}^C\subseteq\mathbb{K}\subseteq\mathbb{F}^{\mathbf N}$, and $M^C$ and $M^{\bar C}$, being $\mathbb{F}^{\mathbf N}$-martingales adapted to $\mathbb{K}$, are square-integrable $\mathbb{K}$-martingales. We define
\begin{align*}
L^2_C(\mathbb{K})&:=\Big\{U:[0,T]\times\Omega\to\mathbb{R}\ \ \mathcal{P}(\mathbb{K})\text{-measurable}:\ \mathbb{E}\int_0^T|U_s|^2\lambda_s\,ds<\infty\Big\},\\
L^2_{\bar C}(\mathbb{K})&:=\Big\{\bar U:[0,T]\times\Omega\to\mathbb{R}\ \ \mathcal{P}(\mathbb{K})\text{-measurable}:\ \mathbb{E}\int_0^T|\bar U_s|^2\bar\lambda_s\,ds<\infty\Big\},
\end{align*}
identified up to $\lambda_s\,ds\,\mathbb{P}(d\omega)$-, respectively $\bar\lambda_s\,ds\,\mathbb{P}(d\omega)$-a.e.\ equality. Since the indicators $\mathbf{1}_{\{\theta\le\lambda_s\}}$ and $\mathbf{1}_{\{\lambda_s<\theta\le\bar\mu\}}$ have disjoint supports, the map
\[
(U,\bar U)\longmapsto J(U,\bar U)_{s,\theta}:=U_s\mathbf{1}_{\{\theta\le\lambda_s\}}+\bar U_s\mathbf{1}_{\{\lambda_s<\theta\le\bar\mu\}}
\]
is a linear isometry from $L^2_C(\mathbb{K})\times L^2_{\bar C}(\mathbb{K})$, equipped with the norm $\|U\|^2_{L^2_C(\mathbb{K})}+\|\bar U\|^2_{L^2_{\bar C}(\mathbb{K})}$, into $L^2(N;\mathbb{F}^{\mathbf N})$, and
$\int_{(0,t]\times\mathbb{R}_+}J(U,\bar U)_{s,\theta}\,\widetilde {\mathbf N}(ds,d\theta)=\int_0^tU_s\,dM^C_s+\int_0^t\bar U_s\,dM^{\bar C}_s$.

\begin{lemma}\label{lem:decomportho}
Let $\lambda$ satisfy Assumption \ref{ass:lambda} and $\lambda_s\le\bar\mu$, $ds\otimes d\mathbb{P}$-a.e., for some constant $\bar\mu>0$.
    Let $F$ be a $\mathcal{K}_T$-measurable random variable. Then
\begin{equation}\label{eq:twochannel}
D_{t,\theta}F=\mathbf{1}_{\{\theta\le\lambda_t\}}D_{t,0}F+\mathbf{1}_{\{\lambda_t<\theta\le\bar\mu\}}D_{t,\bar\mu}F,\qquad d\mathbb{P}\otimes dt\otimes d\theta\text{-a.e.}
\end{equation}
\end{lemma}
\begin{proof}
Since $F$ is $\mathcal K_T$-measurable, we choose a version
$
F(\omega)=f\bigl(C(\omega),\bar C(\omega)\bigr),$
where $f$ is a measurable function of the two counting paths.
All shifted evaluations of $F$, including those at the fixed marks
$0$ and $\bar\mu$, are understood using this version. Recall that $
\lambda_s=\Lambda\bigl(s,C|_{[0,s)}\bigr),$ for some functional $\Lambda$. We set
\[
\omega^{t,\theta}:=\varepsilon^+_{t,\theta}(\omega),\quad 
C^{t,\theta}:=C(\omega^{t,\theta}),
\text{ and }
\bar C^{t,\theta}:=\bar C(\omega^{t,\theta}).
\]
For $d\mathbb P\otimes dt$-almost every $(\omega,t)$, the configuration
$\omega$ has no atom at time $t$. Inserting a point at time $t$ leaves
the history before $t$ unchanged. Hence
\[
C^{t,\theta}|_{[0,t)}=C|_{[0,t)},
\qquad
\lambda_t(\omega^{t,\theta})=\lambda_t(\omega).
\]
We distinguish three cases.

\medskip
\noindent
\textbf{Case 1: $\theta\le\lambda_t(\omega)$.}
Both insertions $(t,\theta)$ and $(t,0)$ are accepted. Thus, they
produce the same accepted history up to time $t$, with one additional
jump at $t$. For $a\in\{\theta,0\}$ and $s\ge t$,
\[
C_s^{t,a}
=
C_{t-}+1+
\int_{(t,s]\times[0,\bar\mu]}
\mathbf 1_{\{u\le\Lambda(r,C^{t,a}|_{[0,r)})\}}
\,\mathbf N(dr,du).
\]
The two processes therefore solve the same thinning equation,
with the same initial history and the same driving Poisson atoms.
Pathwise uniqueness gives
$C^{t,\theta}=C^{t,0}.$\\
Moreover, every Poisson atom in $[0,\bar\mu]$ is either accepted
or rejected. Consequently, for $a\in\{\theta,0\}$,
\[
C_s^{t,a}+\bar C_s^{t,a}
=
\mathbf N\bigl((0,s]\times[0,\bar\mu]\bigr)
+\mathbf 1_{\{t\le s\}},
\qquad s\in[0,T].
\]
Since the accepted processes coincide, the rejected processes also
coincide. Therefore,
\[
(C^{t,\theta},\bar C^{t,\theta})
=
(C^{t,0},\bar C^{t,0}),
\]
and the representation $F=f(C,\bar C)$ yields
\[
D_{t,\theta}F=D_{t,0}F.
\]

\medskip
\noindent
\textbf{Case 2: $\lambda_t(\omega)<\theta\le\bar\mu$.}
The inserted point is rejected. Indeed, the original path $C$
still solves the thinning equation driven by $\omega^{t,\theta}$,
because the additional atom contributes
\[
\mathbf 1_{\{\theta\le\Lambda(t,C|_{[0,t)})\}}=0.
\]
By pathwise uniqueness, $C^{t,\theta}=C.$
Since the intensity depends only on $C$, it also remains unchanged.
Thus, the only modification is one additional rejected point:
\[
\bar C_s^{t,\theta}
=
\bar C_s+\mathbf 1_{\{t\le s\}},
\qquad s\in[0,T].
\]
On this event, $\lambda_t(\omega)<\bar\mu$, so insertion at
$(t,\bar\mu)$ has exactly the same effect. Hence
\[
(C^{t,\theta},\bar C^{t,\theta})
=
(C^{t,\bar\mu},\bar C^{t,\bar\mu}),
\]
which implies $D_{t,\theta}F=D_{t,\bar\mu}F.$
\medskip

\noindent
\textbf{Case 3: $\theta>\bar\mu$.}
Since $\lambda_t\le\bar\mu$, the inserted point is rejected by $C$.
As in Case~2, pathwise uniqueness implies $C^{t,\theta}=C$.
The inserted point is also outside the strip counted by $\bar C$,
so $\bar C^{t,\theta}=\bar C$. Therefore, $
D_{t,\theta}F=0.$ 
\medskip

\noindent
Combining the three cases gives $
D_{t,\theta}F
=
\mathbf 1_{\{\theta\le\lambda_t\}}D_{t,0}F
+
\mathbf 1_{\{\lambda_t<\theta\le\bar\mu\}}D_{t,\bar\mu}F,
\qquad
d\mathbb P\otimes dt\otimes d\theta\text{-a.e.}.$
\end{proof}
\begin{theorem}[Martingale representation in the joint filtration]\label{thm:joint}
Let $\lambda$ satisfy Assumption \ref{ass:lambda} and $\lambda_s\le\bar\mu$, $ds\otimes d\mathbb{P}$-a.e., for some constant $\bar\mu>0$. Let $Q=(Q_t)_{t\in[0,T]}$ be a square-integrable $\mathbb{K}$-martingale. Then $Q$ is also a square-integrable $\mathbb{F}^{\mathbf N}$-martingale, and there exists a unique pair $(U,\bar U)\in L^2_C(\mathbb{K})\times L^2_{\bar C}(\mathbb{K})$ such that
\begin{equation}\label{eq:repG}
Q_t=Q_0+\int_0^tU_s\,\big(dC_s-\lambda_s\,ds\big)+\int_0^t\bar U_s\,\big(d\bar C_s-\bar\lambda_s\,ds\big),\qquad t\in[0,T],\ \mathbb{P}\text{-a.s.}
\end{equation}
Equivalently,
\[
Q_t=Q_0+\int_{(0,t]\times[0,\bar\mu]}\Big(U_s\mathbf{1}_{\{\theta\le\lambda_s\}}+\bar U_s\mathbf{1}_{\{\theta>\lambda_s\}}\Big)\,\widetilde {\mathbf N}(ds,d\theta).
\]
The integrand $Z$ of $Q$ from  Theorem \ref{th:PRPgrosPoisson} satisfies $Z=J(U,\bar U)$, $d\mathbb{P}\otimes ds\otimes d\theta$-a.e.; in particular $Z=0$ on $\{\theta>\bar\mu\}$ and
\[
U_s=\frac1{\lambda_s}\int_0^{\lambda_s}Z_{s,\theta}\,d\theta,\qquad
\bar U_s=\frac1{\bar\mu-\lambda_s}\int_{\lambda_s}^{\bar\mu}Z_{s,\theta}\,d\theta\quad\text{on }\{\lambda_s<\bar\mu\}.
\]
Moreover, $P:=\int_0^{\cdot}U_s\,dM^C_s$ and $\bar P:=\int_0^{\cdot}\bar U_s\,dM^{\bar C}_s$ are orthogonal square-integrable $\mathbb{K}$-martingales starting at zero, $Q=Q_0+P+\bar P$, and
\[
\mathbb{E}\big[|Q_T-Q_0|^2\big]=\mathbb{E}\big[\int_0^T|U_s|^2\lambda_s\,ds\big]+\mathbb{E}\big[\int_0^T|\bar U_s|^2\bar\lambda_s\,ds \big].
\]
\end{theorem}

\begin{proof}
The proof is similar to the proof of Theorem \ref{thm:hawkes-mrp} in the filtration $\mathbb K$.
Let $Q=(P_t)_{t\in[0,T]}$ be a square-integrable
$\mathbb K$-martingale and define the
$\mathbb F^{\mathbf N}$-martingale $R=(R_t)_{t\in[0,T]}$ by

$$
R_t
:=
\E\left[Q_T|\mathcal F_t^{\mathbf N}\right],
\qquad t\in[0,T].
$$
By Theorem~\ref{th:PRPgrosPoisson}, there exists a unique
$Z\in \mathbb L^2(\mathbf N;\mathbb F^{\mathbf N})$ such that
\[
R_t=R_0+
\int_{(0,t]\times\mathbb R_+}
Z_{s,\theta}
\widetilde{\mathbf N}(ds,d\theta),
\qquad t\in[0,T].
\]
Moreover, $
R_0=\E[Q_T]=Q_0.
$ By the
Clark-Ocone identification,

$$
Z_{s,\theta}
=
\E\left[
D_{s,\theta}Q_T
|
\mathcal F_{s-}^{\mathbf N}
\right].
$$
Since $Q_T$ is $\mathcal K_T$-measurable, Lemma \ref{lem:decomportho} together with \eqref{eq:twochannel}  applies to $Q_T$ gives
\[
Z_{t,\theta}=\mathbf{1}_{\{\theta\le\lambda_t\}}\,\mathbb{E}\big[D_{t,0}Q_T\,\big|\,\mathcal{F}^{\mathbf N}_{t-}\big]+\mathbf{1}_{\{\lambda_t<\theta\le\bar\mu\}}\,\mathbb{E}\big[D_{t,\bar\mu}Q_T\,\big|\,\mathcal{F}^{\mathbf N}_{t-}\big],
\]
where the last equality comes from the fact that $\lambda$ is $\mathbb{F}^C$-predictable, hence $\lambda_t$ is $\mathcal{F}^{\mathbf N}_{t-}$-measurable. We then define
\[
U_t:=\mathbb{E}\big[D_{t,0}Q_T\,\big|\,\mathcal{F}^{\mathbf N}_{t-}\big],\qquad
\bar U_t:=\mathbb{E}\big[D_{t,\bar\mu}Q_T\,\big|\,\mathcal{F}^{\mathbf N}_{t-}\big],\qquad t\in[0,T],
\]
two $\mathbb{F}^{\mathbf N}$-predictable processes which do not depend on $\theta$, and obtain
\[
Z_{t,\theta}=U_t\,\mathbf{1}_{\{\theta\le\lambda_t\}}+\bar U_t\,\mathbf{1}_{\{\lambda_t<\theta\le\bar\mu\}}=J(U,\bar U)_{t,\theta},\qquad d\mathbb{P}\otimes dt\otimes d\theta\text{-a.e.}
\]
In particular $Z=0$ on $\{\theta>\bar\mu\}$, and integrating this expression in $\theta$ over $[0,\lambda_t]$, respectively over $(\lambda_t,\bar\mu]$, yields
\[
U_t=\frac1{\lambda_t}\int_0^{\lambda_t}Z_{t,\theta}\,d\theta,\text{ and }
\bar U_t=\frac1{\bar\mu-\lambda_t}\int_{\lambda_t}^{\bar\mu}Z_{t,\theta}\,d\theta\quad\text{on }\{\lambda_t<\bar\mu\}.
\]
Note moreover that $\int_{\mathbb{R}_+}|Z_{t,\theta}|^2d\theta=|U_t|^2\lambda_t+|\bar U_t|^2\bar\lambda_t$, consequently
\[
\mathbb{E}\Big[\int_0^T|U_t|^2\lambda_t\,dt\Big]+\mathbb{E}\Big[\int_0^T|\bar U_t|^2\bar\lambda_t\,dt\Big]=\|Z\|^2_{L^2(\mathbf N;\mathbb{F}^{\mathbf N})}<\infty .
\]
Hence,
\[
R_t=Q_0+\int_0^tU_s\,\big(dC_s-\lambda_s\,ds\big)+\int_0^t\bar U_s\,\big(d\bar C_s-\bar\lambda_s\,ds\big),\qquad t\in[0,T],\ \mathbb{P}\text{-a.s.}\] or equivalently
\[
R_t=Q_0+\int_{(0,t]\times[0,\bar\mu]}\Big(U_s\mathbf{1}_{\{\theta\le\lambda_s\}}+\bar U_s\mathbf{1}_{\{\theta>\lambda_s\}}\Big)\,\widetilde {\mathbf N}(ds,d\theta).
\]
Consequently, $R_t$ is $\mathcal K_t-$measurable and since $
\mathcal K_t\subseteq\mathcal F_t^{\mathbf N},
$ the tower property gives
\begin{align*}
Q_t
&=
\E[Q_T\mid\mathcal K_t]
\\
&=
\E\left[
\E[Q_T\mid\mathcal K_t^{\mathbf N}]
\mid\mathcal K_t
\right]
\\
&=
\E[R_t\mid\mathcal K_t]\\
&=R_t\\
&=
Q_0+\int_0^tU_s\,\big(dC_s-\lambda_s\,ds\big)+\int_0^t\bar U_s\,\big(d\bar C_s-\bar\lambda_s\,ds\big).
\end{align*}
In particular, $Q$ is an $\mathbb F^{\mathbf N}-$martingale. We now turn to the uniqueness of the representation. Suppose \eqref{eq:repG} holds for two pairs $(U,\bar U)$ and $(U',\bar U')$ in $L^2_C(\mathbb{K})\times L^2_{\bar C}(\mathbb{K})$. Subtracting the two representations at $t=T$ and using the thinning identities for $M^C$ and $M^{\bar C}$
\[
0=\int_0^T(U_s-U'_s)\,dM^C_s+\int_0^T(\bar U_s-\bar U'_s)\,dM^{\bar C}_s
=\int_{(0,T]\times\mathbb{R}_+}J(U-U',\bar U-\bar U')_{s,\theta}\,\widetilde {\mathbf N}(ds,d\theta).
\]
By the It\^o isometry for $\widetilde {\mathbf N}$ and since  $\mathbf{1}_{\{\theta\le\lambda_s\}}\mathbf{1}_{\{\lambda_s<\theta\le\bar\mu\}}=0$,
\[
0=\mathbb{E}\int_0^T\!\!\int_{\mathbb{R}_+}\big|J(U-U',\bar U-\bar U')_{s,\theta}\big|^2\,d\theta\,ds
=\mathbb{E}\int_0^T|U_s-U'_s|^2\lambda_s\,ds+\mathbb{E}\int_0^T|\bar U_s-\bar U'_s|^2\bar\lambda_s\,ds .
\]
Hence $U=U',\;$ $\lambda_s\,ds\,\mathbb{P}(d\omega)$-a.e.\ and $\bar U=\bar U',\; $ $\bar\lambda_s\,ds\,\mathbb{P}(d\omega)$-a.e.

\end{proof}

\begin{remark}
Alternatively, we could introduce the unmarked counting process $N$ defined by $N_t=C_t+\bar C_t=\int_{\theta \in [0, \bar \mu]}  {\bf N}((0,t], d \theta) $ and derive a martingale representation with respect to its own filtration $\mathbb F^N$. This filtration is  included in $\mathbb K$. In this case, any $\mathbb F^N-$martingale can be decomposed as in \eqref{eq:repG} with $U = \bar U$.
\end{remark}

\footnotesize
\bibliographystyle{plain}

\begin{thebibliography}{10}

\bibitem{boel1975martingales}
R.~Boel, P.~Varaiya, and E.~Wong.
\newblock Martingales on jump processes. {I}: Representation results.
\newblock {\em SIAM Journal on control}, 13(5):999--1021, 1975.

\bibitem{bremaud1996stability}
P.~Br{\'e}maud and L.~Massouli{\'e}.
\newblock Stability of nonlinear {H}awkes processes.
\newblock {\em The Annals of Probability}, pages 1563--1588, 1996.

\bibitem{ching2006representation}
C.~Ching-Sung and P.A. Meyer.
\newblock Sur la repr{\'e}sentation des martingales comme int{\'e}grales
  stochastiques dans les processus ponctuels.
\newblock In {\em S{\'e}minaire de Probabilit{\'e}s IX Universit{\'e} de
  Strasbourg}, pages 226--236. Springer, 2006.

\bibitem{confortola2014backward}
F.~Confortola and M.~Fuhrman.
\newblock Backward stochastic differential equations associated to jump
  {M}arkov processes and applications.
\newblock {\em Stochastic Processes and their Applications}, 124(1):289--316,
  2014.

\bibitem{daley2008introduction}
D.J. Daley and D.~Vere-Jones.
\newblock {\em An introduction to the theory of point processes: volume {II}:
  general theory and structure}.
\newblock Springer, 2008.

\bibitem{Hawkes}
A.~G. Hawkes.
\newblock Spectra of some self-exciting and mutually exciting point processes.
\newblock {\em Biometrika}, 58(1):83--90, 1971.

\bibitem{hawkes1974cluster}
A.G. Hawkes and D.~Oakes.
\newblock A cluster process representation of a self-exciting process.
\newblock {\em Journal of Applied Probability}, 11(3):493--503, 1974.

\bibitem{HHKR}
C.~Hillairet, L.~Huang, M.~Khabou, and A.~R{\'e}veillac.
\newblock {\em ALEA, Lat. Am. J. Probab. Math. Stat 19}, pages 1293--1328.

\bibitem{hillairet2024poisson}
C.~Hillairet, T.~Peyrat, and A.~R{\'e}veillac.
\newblock A non-compensated {C}lark-{O}cone formula for functionals of counting
  processes.
\newblock {\em ESAIM: PS Volume 29, 2025}, pages 158--183, 2025.

\bibitem{hillairet2024chaotic}
C.~Hillairet and A.~R{\'e}veillac.
\newblock On the chaotic expansion for counting processes.
\newblock {\em Electronic Journal of Probability}, 29:1--33, 2024.

\bibitem{jacod1975multivariate}
J.~Jacod.
\newblock Multivariate {P}oint processes: predictable projection,
  {R}adon-{N}ikodym derivatives, representation of martingales.
\newblock {\em Zeitschrift f{\"u}r Wahrscheinlichkeitstheorie und verwandte
  Gebiete}, 31(3):235--253, 1975.

\bibitem{jacod1975MRP}
J.~Jacod.
\newblock Multivariate {P}oint processes: predictable projection,
  {R}adon-{N}ikodym derivatives, representation of martingales.
\newblock {\em Zeitschrift f{\"u}r Wahrscheinlichkeitstheorie und verwandte
  Gebiete}, 31(3):235--253, 1975.

\bibitem{jacod2013limit}
J.~Jacod and A.~Shiryaev.
\newblock {\em Limit theorems for stochastic processes}.
\newblock Springer Science \& Business Media, 2003.

\bibitem{Last2016}
G.~Last.
\newblock {\em Stochastic Analysis for Poisson Point Processes. In
  \textit{Stochastic Analysis for Poisson Point Processes: Malliavin Calculus,
  Wiener-It{\^o} Chaos Expansions and Stochastic Geometry} (Editors: {P}eccati,
  {G}. and {R}eitzner, M.)}, pages 1--36.
\newblock Springer International Publishing, 2016.

\bibitem{last2017penrose}
G.~Last and M.~Penrose.
\newblock Lectures on the {P}oisson process, 2017.

\bibitem{ogata1981lewis}
Y.~Ogata.
\newblock On {L}ewis' simulation method for {P}oint processes.
\newblock {\em IEEE transactions on information theory}, 27(1):23--31, 1981.

\bibitem{GapeevJeanblancWu2025}
G.V. Pavel, M.~Jeanblanc, and D.~Wu.
\newblock Projections in enlargements of filtrations under {Jacod}'s absolute
  continuity hypothesis for marked point processes.
\newblock {\em Journal of Theoretical Probability}, 38(4), 2025.

\bibitem{Privault_2009}
N.~Privault.
\newblock {\em Stochastic Analysis in Discrete and Continuous Settings}.
\newblock Springer Berlin Heidelberg, 2009.

\end{thebibliography}

\end{document}